\documentclass[a4paper,oneside,11pt]{amsart}

\reversemarginpar

\usepackage[colorlinks,hyperindex,linkcolor=blue,urlcolor=black,pdftitle={Notes on the codimension one conjecture in the operator corona theorem},pdfauthor={Maria F.Gamal'}]
{hyperref}

\usepackage{amsmath}
\usepackage{amsfonts}
\usepackage{amssymb}
\usepackage{amsthm}

\usepackage[english]{babel}
\usepackage{enumerate}

\usepackage{graphicx}
\usepackage{color}

\usepackage[abbrev]{amsrefs}

\numberwithin{equation}{section}

\theoremstyle{plain}
\newtheorem{theorem}{Theorem}[section]
\newtheorem{lemma}[theorem]{Lemma}
\newtheorem{corollary}[theorem]{Corollary}

\newtheorem{theoremcite}{Theorem}

\theoremstyle{definition}

\begin{document}

\title[Notes on the codimension one conjecture]{Notes on the codimension one conjecture in the operator corona theorem}

\author{Maria F. Gamal'}
\address{
 St. Petersburg Branch\\ V. A. Steklov Institute 
of Mathematics\\
 Russian Academy of Sciences\\ Fontanka 27, St. Petersburg\\ 
191023, Russia  
}
\email{gamal@pdmi.ras.ru}

\thanks{Partially supported by RFBR grant No. 14-01-00748-a}



\begin{abstract}
Answering on the question of S.R.Treil \cite{23}, for every $\delta$, $0<\delta<1$, 
examples of contractions are constructed such that their characteristic functions 
$F\in H^\infty(\mathcal E\to\mathcal E_\ast)$ satisfy the conditions $$\|F(z)x\|\geq\delta\|x\| \  \text{ and } \ 
\dim\mathcal E_\ast\ominus F(z)\mathcal E =1 \ \text{  for every } \ z\in\mathbb D, \ \ x\in\mathcal E,$$
 but $F$ are not
left invertible. Also, it is shown that the condition 
$$\sup_{z\in\mathbb D}\|I-F(z)^\ast F(z)\|_{\frak S_1}<\infty,$$ 
where $\frak S_1$ is the trace class of operators, which is sufficient for the 
left invertibility 
of the operator-valued function $F$ satisfying the estimate $\|F(z)x\|\geq\delta\|x\|$ 
for every $z\in\mathbb D$, $x\in\mathcal E$, with some $\delta>0$ (S.R.Treil,  \cite{22}), is necessary 
for the left invertibility of an inner function $F$ such that $\dim\mathcal E_\ast\ominus F(z)\mathcal E<\infty$ 
for some $z\in\mathbb D$.
\end{abstract} 

\maketitle

\section{Introduction}

Let $\mathbb D$ be the open unit disk, let $\mathbb T$  be the unit circle,
 and let $\mathcal E$ and $\mathcal E_\ast$ be separable Hilbert spaces.
The space $H^\infty(\mathcal E\to\mathcal E_\ast)$ is the space of bounded analytic functions on $\mathbb D$ 
whose values are (linear, bounded) operators acting from $\mathcal E$ to $\mathcal E_\ast$.
If $F\in H^\infty(\mathcal E\to\mathcal E_\ast)$, then $F$ has  nontangential boundary values 
$F(\zeta)$ for a.e. $\zeta\in\mathbb T$ with respect to the Lebesgue measure $m$ on $\mathbb T$.
Every function  $F\in H^\infty(\mathcal E\to\mathcal E_\ast)$ has the inner-outer 
factorization, that is, there exist an auxiliary Hilbert space $\mathcal D$ and  two functions 
$\Theta\in H^\infty(\mathcal E\to\mathcal D)$ and $\Omega\in H^\infty(\mathcal D\to\mathcal E_\ast)$ such that 
$F=\Theta\Omega$, $\Theta$ is inner, that is, $\Theta(\zeta)$ is an isometry for a.e. $\zeta\in\mathbb T$, 
and $\Omega$ is outer. The definition of outer function is not recalled here,
 but it need to mentioned that 
for an outer function $\Omega$, $\operatorname{clos}\Omega(z)\mathcal D=\mathcal E_\ast$
for all $z\in\mathbb D$ and for a.e. $z\in\mathbb T$. Recall that a function $F\in H^\infty(\mathcal E\to\mathcal E_\ast)$ is called $\ast$-inner
($\ast$-outer), if  the function $F_\ast\in H^\infty(\mathcal E_\ast\to\mathcal E)$, 
$F_\ast(z)=F^\ast(\overline z)$, $z\in\mathbb D$, is inner (outer) (see {\cite[Ch.V]{16}}, also {\cite[\S A.3.11.5]{14}}).

Every analytic operator-valued function $F\in H^\infty(\mathcal E\to\mathcal E_\ast)$ such that $\|F\|\leq 1$ can be represented as an orthogonal sum of a unitary constant and a purely contractive function $F_0$, that is, 
there exist representations of Hilbert spaces $\mathcal E=\mathcal E'\oplus\mathcal E_0$ and 
$\mathcal E_\ast=\mathcal E_\ast'\oplus\mathcal E_{\ast 0}$ and a unitary operator 
$W\colon\mathcal E'\to\mathcal E_\ast'$ such that $F(z)\mathcal E'\subset\mathcal E_\ast'$, 
$F(z)|_{\mathcal E'}=W$, $F(z)\mathcal E_0\subset\mathcal E_{\ast 0}$, $F_0(z)=F(z)|_{\mathcal E_0}$ for every 
$z\in\mathbb D$, and $\|F_0(0)x\|<\|x\|$ for every $x\in\mathcal E_0$, $x\neq 0$ 
({\cite[Proposition V.2.1]{16}}). In all questions considered in this note it can be supposed that $\|F\|\leq 1$ and $F$ is purely contractive.

The Operator Corona Problem is to find necessary and sufficient condition 
for a function   
$F\in H^\infty(\mathcal E\to\mathcal E_\ast)$ to be left invertible, that is, to exist a function 
$G\in H^\infty(\mathcal E_\ast\to\mathcal E)$ such that $G(z)F(z)=I_{\mathcal E}$ for all $z\in\mathbb D$.
If $F\in H^\infty(\mathcal E\to\mathcal E_\ast)$ is left invertible, then there exists $\delta>0$
such that 
\begin{equation}\label{1.1}\|F(z)x\|\geq \delta \|x\| \ \ \ \ \text{ for all } \ x\in\mathcal E, \ \ z\in\mathbb D. \end{equation}
It is easy to see that if $F$ satisfies \eqref{1.1} and $F=\Theta\Omega$ is the inner-outer factorization of $F$,
then the outer function $\Omega$ is  invertible, and the inner function $\Theta$ satisfies \eqref{1.1} 
(may be with another $\delta$).

The condition \eqref{1.1} is sufficient for left invertibility, if $\dim\mathcal E<\infty$ \cite{19}, but 
is not sufficient in general \cite{20}, \cite{21}. Also, \eqref{1.1} is not sufficient under additional assumption 
\begin{equation}\label{1.2}\dim\mathcal E_\ast\ominus F(z)\mathcal E=1 \ \ \text{ for all } \ z\in\mathbb D. \end{equation}
In \cite{23}, for every $\delta$, $0<\delta<1/3$, two functions $F_1$, $F_2\in H^\infty(\mathcal E\to\mathcal E_\ast)$ 
are constructed such that 
\begin{equation}\label{1.3}  \|x\|\geq\|F_k(z)x\|\geq \delta \|x\| \ \ \ \ \text{ for all } \ x\in\mathcal E, \  \ z\in\mathbb D,\end{equation}
\eqref{1.2} is fulfilled for $F_k$, $k=1,2$, $$F_1(\zeta)\mathcal E=\mathcal E_\ast \ \text{  for a.e. }\  \zeta\in\mathbb T,
\ \ \ \ 
\dim\mathcal E_\ast\ominus F_2(\zeta)\mathcal E=1 \ \text{ for a.e. }\  \zeta\in\mathbb T,$$ 
but $F_1$ and  $F_2$ are not 
left invertible. It is mentioned in \cite{23} that the method from  \cite{20}, \cite{21} gives examples of such functions for
$\delta<1/\sqrt 2$, and a question was posed if for every $\delta$, $0<\delta<1$, there exists 
$F\in H^\infty(\mathcal E\to\mathcal E_\ast)$ such that $F$ satisfies \eqref{1.2} and \eqref{1.3}, and $F$ is not 
left invertible.
In this note, it is shown that such function $F$ exists for every $0<\delta<1$, and 
 any from the following cases can be realized: 
\begin{equation}\label{1.4}F(\zeta)\mathcal E=\mathcal E_\ast \ \ \text{ for a.e. } \ \zeta\in\mathbb T,  \end{equation}
 or  
\begin{equation}\label{1.5}\dim\mathcal E_\ast\ominus F(\zeta)\mathcal E=1 \ \ \text{ for a.e. } \ \zeta\in\mathbb T,  \end{equation}
 or 
\begin{equation}\label{1.6}\dim\mathcal E_\ast\ominus F(\zeta)\mathcal E=1 \ \ \text{ for a.e. } \ \zeta\in E,  \ \text{  and }
\ \  
F(\zeta)\mathcal E=\mathcal E_\ast \ \ \text{ for a.e. } \zeta\in\mathbb T\setminus E, \end{equation}
 where 
$E\subset\mathbb T$ is a closed set satisfying the Carleson condition with $0<m(E)<1$ (see Sec. 5 
of this note where  the definition is recalled).

Actually, not operator-valued functions, but contractions are constructed, and the required 
functions are the characteristic functions of these contractions {\cite[Ch. VI]{16}}, see Sec. 3 of this note.

In \cite{22}, some sufficient conditions are given, 
which imply the left invertibility of functions, in particular, 
it is proved in  \cite{22}, that if $F\in H^\infty(\mathcal E\to\mathcal E_\ast)$ 
satisfies to \eqref{1.1} and 
\begin{equation}\label{1.7}\sup_{z\in\mathbb D}\|I_{\mathcal E}-F(z)^\ast F(z)\|_{\frak S_1}<\infty, 
\end{equation}
where $\frak S_1$ is the trace class of operators, then $F$ is left invertible.
In this note, it is shown that the condition \eqref{1.7} is necessary for left invertibility of $F$, if $F$ is inner
and $\dim\mathcal E_\ast\ominus F(z)\mathcal E<\infty$ for some $z\in\mathbb D$. Actually, an appropriate fact 
is proved for contractions with such characteristic functions, and the statement on function follows from 
the fact for contractions.

 The function $F\in H^\infty(\mathcal E\to\mathcal E_\ast)$ has a left scalar multiple  if there exist 
$G\in H^\infty(\mathcal E_\ast\to\mathcal E)$ and a function $\rho\in H^\infty$, where $H^\infty$ is the algebra of all
bounded analytic functions on $\mathbb D$,  such that 
$\rho(z)I_{\mathcal E}= G(z)F(z)$ for all $z\in\mathbb D$. 
 The left invertibility of $F$ means that $F$ has a scalar multiple, which is invertible in $H^\infty$. 
Functions $F_1$ and  $F_2$ from  \cite{23} mentioned above do not have scalar multiple,
the proof is actually the same as the proof that $F_1$ and $F_2$ are not left invertible. 
In this note, it is shown that the existence of the left scalar multiple of $F$ with \eqref{1.1} is not sufficient 
for the left invertibility of $F$, even if $F$ is inner and $F$ satisfies \eqref{1.2}. In this case, 
$I_{\mathcal E}-F(z)^\ast F(z)\in\frak S_1$ for every $z\in\mathbb D$, but 
$\sup_{z\in\mathbb D}\|I_{\mathcal E}-F(z)^\ast F(z)\|_{\frak S_1}=\infty$. Again, 
an appropriate
contraction is constructed, and $F$ is the characteristic function of this contraction. 

We shall use the following notation: $\mathbb D$ is the open unit disk, 
$\mathbb T$ is the unit circle, $m$ is the normalized Lebesgue measure on $\mathbb T$, 
and $H^2$ is 
the Hardy space in $\mathbb D$. 
For a  positive integer $n$, $1\leq n<\infty$, $H^2_n$ and $L^2_n$ 
are orthogonal sums of $n$ copies of spaces $H^2$ and $L^2 =L^2(\mathbb T,m)$, 
respectively. 
The unilateral shift $S_n$ and the bilateral shift $U_n$ of multiplicity $n$ 
 are the operators  of multiplication by the independent variable in spaces 
 $H^2_n$ and $L^2_n$, respectively. For a Borel set $\sigma\subset\mathbb T$, by 
$U(\sigma)$ we denote the operator of multiplication by the
independent variable on the space $L^2(\sigma,m)$ of functions from $L^2$ 
that are equal to zero a.e. on $\mathbb T\setminus\sigma$.

For a Hilbert space $\mathcal H$, by $I_{\mathcal H}$ and $\mathbb O_{\mathcal H}$ the identity and the zero operators 
acting on $\mathcal H$ are denoted, respectively.

Let $T$ and $R$ be operators on spaces $\mathcal H$ and $\mathcal K$, 
respectively, and let $X:\mathcal H\to\mathcal K$ be a (linear, bounded) operator which 
 intertwines $T$ and $R$:  $XT=RX$. If $X$ is unitary, then $T$ and $R$ 
are called {\it unitarily equivalent}, in notation: 
$T\cong R$. If $X$ is invertible (the inverse $X^{-1}$  is bounded), then $T$ and $R$ 
are called {\it similar}, in notation: $T\approx R$. 
If $X$ a quasiaffinity, that is, $\ker X=\{0\}$ and 
$\operatorname{clos}X\mathcal H=\mathcal K$, then 
$T$ is called a {\it quasiaffine transform} of $R$, 
in notation: $T\prec R$. If $T\prec R$ and $R\prec T$, 
then $T$ and $R$ are called {\it quasisimilar}, in notation: $T\sim R$. 

Let  $\mathcal H$ be a Hilbert space, 
and let $T\colon\mathcal H\to\mathcal H$ be a (linear, bounded) operator. $T$ is called a {\it contraction}, if $\|T\|\leq 1$.
Let $T$ be a contraction
 on a space $\mathcal H$. $T$ is {\it of class} $C_{1\cdot}$ 
 ($T\in C_{1\cdot}$), if $\lim_{n\to\infty}\|T^nx\|>0$ for
each $x\in \mathcal H$, $x\neq 0$, $T$ is {\it of class} $C_{0\cdot}$ 
($T\in C_{0\cdot}$), if 
$\lim_{n\to\infty}\|T^nx\|=0$ for each $x\in\mathcal H$, and $T$ is 
of class $C_{\cdot a}$, $a=0,1$, if
$T^\ast$ is of class $C_{a\cdot}$. 

 It is easy to see that if a contraction $T$ 
is a quasiaffine transform of an isometry, then $T$ is of class $C_{1\cdot}$, 
 and if $T$ is  a quasiaffine transform of a unilateral shift,  
then $T$ is of class $C_{10}$.

The paper is organized as follows.
In Sec. 2 and 3, the known facts about contractions, their relations to isometries, 
and their characteristic functions are collected. In Sec. 4, the necessity of \eqref{1.7} 
to the left invertibility of some operator-valued functions is proved. Sec. 5 is the main section 
of this paper, where for any $\delta$, $0<\delta<1$, examples of subnormal contractions are constructed such that 
their characteristic functions satisfy \eqref{1.2}, and \eqref{1.3} with $\delta$,  and are not 
left invertible.

\section{Isometric and unitary asymptotes of contractions}

For a contraction $T$ the {\it isometric asymptote} $(X_{T,+},T_+^{(a)})$, and  
the {\it unitary asymptote} $(X_T,T^{(a)})$ was defined, 
 see, for example, {\cite[Ch. IX.1]{16}}. There are some ways to construct
the isometric asymptote of the contraction, for our purpose, 
it is convenient to use the following, see \cite{11}.  
 
Let $(\cdot,\cdot)$ be the inner 
product on the Hilbert space $\mathcal H$, 
and let $T\colon\mathcal H\to\mathcal H$ be a  contraction.
 Define a new semi-inner product on
$\mathcal H$ by the formula $\langle x,y\rangle=\lim_{n\to\infty}(T^nx,T^ny)$,
 where $x,y\in \mathcal H$. Set 
$$\mathcal H_0=\mathcal H_{T,0}=\{x\in\mathcal H:\ \langle x, x\rangle=0\}.$$
Then the factor space $\mathcal H/\mathcal H_0$ with the inner product 
$\langle x+\mathcal H_0, y+\mathcal H_0\rangle=\langle x,y\rangle$ will 
be an inner product space.
Let $\mathcal H_+^{(a)}$ denote the resulting Hilbert space obtained 
by completion, and let $X_{T,+}\colon\mathcal H\to\mathcal H_+^{(a)}$ 
be the natural imbedding,
$X_{T,+}x=x+\mathcal H_0$. Clearly, $X_{T,+}$ is a (linear, bounded) operator, and $\|X_{T,+}\|\leq 1$. 
Clearly, $\langle Tx,Ty\rangle=\langle x,y\rangle$ for every $x,y\in \mathcal H$.
Therefore, $T_1: x+\mathcal H_0\mapsto Tx+\mathcal H_0$ is a well-defined isometry on 
$\mathcal H/\mathcal H_0$. Denote by $T_+^{(a)}$ the continuous extension of $T_1$ to the space
$\mathcal H_+^{(a)}$. Clearly, $X_{T,+}T=T_+^{(a)}X_{T,+}$. The pair $(X_{T,+},T_+^{(a)})$ is called 
the {\it isometric asymptote} of a contraction $T$. The operator $X_{T,+}$ is called the {\it canonical intertwining mapping}.

A contraction $T$ is similar 
to an isometry $V$ if and only if $X_{T,+}$ is an invertible operator, 
that is, $\ker X_{T,+}=\{0\}$ and $X_{T,+}\mathcal H= \mathcal H_+^{(a)}$, and in this case, $V\cong T_+^{(a)}$
(see {\cite[Theorem 1]{11}}). In particular, if $T$ is a contraction of class $C_{10}$, and $T_+^{(a)}$ 
is a unitary operator, then $T$ is not similar to an isometry. 

Denote by $T^{(a)}$ the minimal unitary extension of $T_+^{(a)}$, by $\mathcal H^{(a)}\supset\mathcal H_+^{(a)}$ 
the space on which  $T^{(a)}$ acts, and by $X_T$ the imbedding of $\mathcal H$ into $\mathcal H^{(a)}$.
Clearly, $X_TT=T^{(a)}X_T$ and $X_{T,+}x=X_Tx$ for every $x\in\mathcal H$.
The pair $(X_T,T^{(a)})$ is called the {\it unitary asymptote} $(X,T^{(a)})$ of a contraction $T$. 

\section{Contractions and their characteristic functions}

All statement of this section are well-known and can be found in {\cite[Ch. VI]{16}}, see also 
{\cite[Ch. C.1]{14}}.

Let $\mathcal H$ be a separable Hilbert space, and let $T\colon\mathcal H\to\mathcal H$ be 
a  contraction.
A contraction $T$ is called {\it completely nonunitary}, if $T$ has no invariant subspace such that 
the restriction of $T$ on this subspace is unitary.
For a contraction $T$ put $\mathcal D_T=\operatorname{clos}(I_{\mathcal H}-T^\ast T)\mathcal H$.
It is easy to see that \begin{equation}\label{3.1}\text{ if } \ x\in\mathcal H\ominus\mathcal D_T, \ \  
\text{ then } \ \ x=T^\ast Tx \ \ \text{ and } \ \ \|Tx\|=\|x\|. \end{equation}  
Also, $T\mathcal D_T\subset\mathcal D_{T^\ast}$ and $T(\mathcal H\ominus\mathcal D_T)=
 \mathcal H\ominus\mathcal D_{T^\ast}$ (see {\cite[Ch. I.3.1]{16}}), therefore, 
\begin{equation}\label{3.2} \dim \mathcal D_{T^\ast}\ominus T\mathcal D_T = \dim\mathcal H\ominus T\mathcal H .  \end{equation}  
Since $I_{\mathcal H} - T^\ast T=(I_{\mathcal D_T} - T^\ast T|_{\mathcal D_T})\oplus\mathbb O_{\mathcal H\ominus\mathcal D_T}$,
\begin{equation}\label{3.3}\|I_{\mathcal H} - T^\ast T\|_{\frak S_1}=\|I_{\mathcal D_T} - T^\ast T|_{\mathcal D_T}\|_{\frak S_1}.   \end{equation}  

\begin{lemma}\label{lem3.1} Let $T\colon\mathcal H\to\mathcal H$ be a contraction, and let $0<\delta\leq 1$.
Then $\|Tx\|\geq\delta\|x\|$ for every $x\in\mathcal H$ if and only if 
$\|Tx\|\geq\delta\|x\|$ for every $x\in\mathcal D_T$.\end{lemma}
\begin{proof} Indeed, it need to prove the ``if" part only. Let $x\in\mathcal D_T$, and let $y\in\mathcal H\ominus\mathcal D_T$.
Then, by \eqref{3.1},  $(Tx,Ty)=(x,T^\ast Ty)=(x,y)=0$ and $\|Ty\|=\|y\|\geq\delta\|y\|$.
Therefore, $\|T(x+y)\|^2 = \|Tx\|^2+\|Ty\|^2\geq\delta^2\|x\|^2+\delta^2\|y\|^2=\delta^2\|x+y\|^2$. \end{proof}

The characteristic function $\Theta_T$ of the contaction $T$ is the analytic 
operator-valued function acting by the formula
\begin{equation}\label{3.4} 
\Theta_T(z)=\big(-T+z (I-TT^\ast)^{1/2}(I-zT^\ast)^{-1}(I-T^\ast T)^{1/2}\big)|_{\mathcal D_T}, 
 \ \ z\in\mathbb D. \end{equation}  
For every $z\in\mathbb D$ the inclusion $\Theta_T(z)\mathcal D_T\subset\mathcal D_{T^\ast}$ holds,
the mapping $z\mapsto\Theta_T(z)$ is an analytic function from $\mathbb D$ to
the space of all (linear, bounded) operators from $\mathcal D_T$ to $\mathcal D_{T^\ast}$,
and $\|\Theta_T(z)\|\leq 1$ for every $z\in\mathbb D$. 
That is, $\Theta_T\in H^\infty(\mathcal D_T\to\mathcal D_{T^\ast})$, and $\|\Theta_T\|\leq 1$. It is easy to see 
that $\Theta_T$ is  purely contractive.
Conversely, for every analytic operator-valued function $F\in H^\infty(\mathcal E\to\mathcal E_\ast)$
 such that  $\|F\|\leq 1$ and $F$ is purely contractive there exists a contraction $T$ such that $F=\Theta_T$ {\cite[Ch. VI]{16}}. 

\smallskip

The following theorem was proved in \cite{15}, see also {\cite[C.1.5.5]{14}}.

\begin{theoremcite}\cite{15}\label{tha}
 The contraction $T$ is similar to an isometry if and only if 
$\Theta_T$ is left invertible.\end{theoremcite}

 Let $T$ be a completely nonunitary contraction. Then $T$ is
  of class $C_{1\cdot}$ if and only if $\Theta_T$ is $\ast$-outer, 
  and $T$ is of class $C_{\cdot 0}$,
if and only if $\Theta_T$ is inner {\cite[VI.3.5]{16}}. 

Recall that the {\it multiplicity} of an operator is the minimum dimension
of its reproducing subspaces. An operator is called {\it cyclic} if its multiplicity is equal to $1$.

\smallskip

The following theorem was proved in \cite{7}, \cite{10}, \cite{17}, \cite{24}, \cite{25}.

\begin{theoremcite}\label{thb} Let $T$ be a contraction, and let 
$1\leq n <\infty$. The following are equivalent:

$(1)$ $T\prec S_n$;

$(2)$ $T$ is of class $C_{10}$, $\dim\ker T^\ast =n$, and 
$I-T^\ast T \in\frak S_1$;

$(3)$ $\Theta_T$ is an inner $\ast$-outer function, $\Theta_T$ has a left 
scalar multiple, and $\dim\mathcal D_{T^\ast}\ominus\Theta_T(\lambda)\mathcal D_T=n$ 
for some $\lambda\in\mathbb D$.

Moreover, if $T$ is a contraction such that $T\prec S_n$, $1\leq n <\infty$, then the following are equivalent:

$(4)$ $T\sim S_n$;

$(5)$ multiplicity of $T$ is equal to $n$;

$(6)$ $\Theta_T$ has an outer left scalar multiple.\end{theoremcite}

\smallskip

{\bf Remark.} If $T$ is a contraction and  $T\prec S_n$, $1\leq n <\infty$, then  $b_\lambda(T)$ is a contraction and 
$b_\lambda(T)\prec b_\lambda(S_n)\cong S_n$. Therefore, $I-b_\lambda(T)^\ast b_\lambda(T)\in\frak S_1$ for 
every $\lambda\in\mathbb D$. Here  $b_\lambda(T)=(T-\lambda)(I-\overline\lambda T)^{-1}$.

\smallskip

Let $T$ be a completely nonunitary contraction. 
Put \begin{equation}\label{3.5}\Delta_\ast(\zeta)=(I_{\mathcal D_{T^\ast}}-
\Theta_T(\zeta)\Theta_T(\zeta)^\ast)^{1/2}, \  \  \zeta\in\mathbb T, \ \   
\text{ and }\   \ \omega_T=\{\zeta\in\mathbb T: 
\Delta_\ast(\zeta) \neq \mathbb O\}.\end{equation}
 Then the unitary asymptote $T^{(a)}$ of a completely nonunitary contraction $T$
is unitarily equivalent to the operator of multiplication by the independent variable $\zeta$ on
$\operatorname{clos}\Delta_\ast L^2(\mathcal D_{T^\ast})$. In particular, $T^{(a)}$ is cyclic if and only if 
\begin{equation}\label{3.6}\dim \Delta_\ast(\zeta) \mathcal D_{T^\ast}\leq 1 \ \ \text{ for  a.e. } \ \zeta\in\mathbb T,
\end{equation} and in this case
$T^{(a)}\cong U(\omega_T)$ (see {\cite[Ch. IX.2]{16}}).
Also, if  the function $\Theta_T$ is inner and \eqref{3.6} holds, then 
\begin{equation}\label{3.7}\begin{aligned}\omega_T & = \{ \zeta\in\mathbb T: \  
\dim \mathcal D_{T^\ast}\ominus\Theta(\zeta)\mathcal D_T = 1 \} \\
\text{ and } \ \  \mathbb T\setminus\omega_T & =\{ \zeta\in\mathbb T:
 \  \Theta(\zeta)\mathcal D_T=\mathcal D_{T^\ast}\}.\end{aligned}\end{equation}

For $\lambda\in\mathbb D$ put $b_\lambda(z)=\frac{z-\lambda}{1-\overline\lambda z}$, $z\in\mathbb D$. 
Then $b_\lambda(T)=(T-\lambda)(I-\overline\lambda T)^{-1}$ is a contraction. For every $\lambda\in\mathbb D$ there exists
unitary operators 
$$V_\lambda\colon\mathcal D_T\to\mathcal D_{b_\lambda(T)} \ \ \ \text{ and } \ \ \ 
V_{\lambda\ast}\colon\mathcal D_{b_\lambda(T)^\ast}\to\mathcal D_{T^\ast}$$
such that 
\begin{equation}\label{3.8} V_{\lambda\ast}\Theta_{b_\lambda(T)}(z)V_\lambda=
\Theta_T(b_{-\lambda}(z)). \end{equation}
Setting $z=0$ in \eqref{3.4} and \eqref{3.8},  we conclude that 
\begin{equation}\label{3.9}\Theta_T(\lambda)=-V_{\lambda\ast}b_\lambda(T)V_\lambda \ \ \text{ for every } \ 
\lambda\in\mathbb D \end{equation}
({\cite[Ch. VI.1.3]{16}}).

The following lemma is a straightforward consequence of \eqref{3.2}, \eqref{3.3}, \eqref{3.9}, 
and Lemma \ref{lem3.1}.

\begin{lemma}\label{lem3.2}  Suppose $T\colon\mathcal H\to\mathcal H$ is a completely 
nonunitary contraction.

$\ \ \text{\rm (i)}$  Let $0<\delta\leq1$. Then $\|\Theta_T(\lambda)x\|\geq\delta\|x\|$ 
for every $x\in\mathcal D_T$, $\lambda\in\mathbb D$, 
if and only if $\|b_\lambda(T)x\|\geq\delta\|x\|$ for every 
$x\in\mathcal H$, $\lambda\in\mathbb D$.

$\ \text{\rm (ii)}$ $\dim\mathcal D_{T^\ast}\ominus\Theta_T(\lambda)\mathcal D_T=
\dim\mathcal H\ominus b_\lambda(T)\mathcal H$ 
for every $\lambda\in\mathbb D$.

$\text{\rm (iii)}$  $\|I_{\mathcal D_T}-
\Theta_T^\ast(\lambda)\Theta_T(\lambda)\|_{\frak S_1} = 
\|I_{\mathcal H}-b_\lambda(T)^\ast b_\lambda(T) \|_{\frak S_1}$
for every $\lambda\in\mathbb D$.\end{lemma}

\section{On contractions similar to an isometry}

The following theorem is actually proved in {\cite[Theorem 2.1]{7}} (see also enlarged version on arXiv).

\begin{theorem}\label{th4.1} \cite{7} Suppose $T$ is a contraction 
with finite multiplicity, 
 and $T$ is similar to an isometry.
Then \begin{equation}\label{4.1}\sup_{\lambda\in\mathbb D}\|I-b_\lambda(T)^\ast b_\lambda(T)\|_{\frak S_1}
<\infty. \end{equation} \end{theorem}

\begin{corollary}\label{cor4.2}  Suppose $\mathcal E$, $\mathcal E_\ast$ are Hilbert spaces, 
$F\in H^\infty(\mathcal E\to\mathcal E_\ast)$ is an inner function, and 
$\dim \mathcal E_\ast\ominus F(\lambda)\mathcal E<\infty$ for some $\lambda \in\mathbb D$.
If $F$ is left invertible, then $F$ satisfies \eqref{1.7}.\end{corollary}

\begin{proof} Let $\mathcal H$ be a Hilbert space, and let $T\colon\mathcal H\to\mathcal H$ be 
a contraction such that $\Theta_T=F$, where $\Theta_T$ is the characteristic function 
of $T$ (see {\cite[Ch. VI.3]{16}}). Since $F$ is inner, $T\in C_{\cdot 0}$, see {\cite[VI.3.5]{16}}. 
 By Lemma \ref{lem3.2}(ii), 
$$\dim \mathcal H\ominus b_\lambda(T)\mathcal H = \dim \mathcal E_\ast\ominus F(\lambda)\mathcal E<\infty.$$
Now suppose that $F$ is left invertible. Then, by Theorem \ref{tha}, there exist a Hilbert space 
$\mathcal K$ and an  isometry $V\colon\mathcal K\to\mathcal K$ such that $T\approx V$. Since $T\in C_{\cdot 0}$, 
$V$ is a unilateral shift, and the multiplicity of $V$ is equal to
$$\dim \mathcal K\ominus b_\lambda(V)\mathcal K = \dim \mathcal H\ominus b_\lambda(T)\mathcal H <\infty.$$
By Theorem \ref{th4.1}, $T$ satisfies \eqref{4.1}. By Lemma \ref{lem3.2}(iii),  $F$ satisfies \eqref{1.7}.  \end{proof} 

\section{Subnormal contractions}

Operators that are considered in this sections are subnormal ones, and are studied by 
many authors, the reader can consult with   the book \cite{6}.

 Let $\nu$ be
 a positive finite Borel measure on the closed unit disk $\overline{\mathbb D}$. Denote by 
$P^2(\nu)$ the closure of analytic polynomials in $L^2(\nu)$, and by $S_\nu$
the operator of multiplication by the independent variable in $P^2(\nu)$, i.e.
\begin{align*} & S_\nu: P^2(\nu)\to P^2(\nu), \\ & (S_\nu f)(z)=zf(z) \text{ for a.e. } z\in \overline{\mathbb D}
\ \text{ with respect to } \nu , \ \ f\in P^2(\nu).\end{align*}
Clearly, $S_\nu$ is a contraction.  

Recall that $m$ is the Lebesgue measure on $\mathbb T$. If $\nu=m$, then $S_\nu$ 
is the unilateral shift of multiplicity 1, it is denoted by $S$ in this section.

The following lemma is a straightforward consequence of the construction of the 
isometric asymptote of a contraction from \cite{11}, see Sec. 2 of this paper, so its proof is omitted.

\begin{lemma}\label{lem5.1} Suppose $\nu$ is a positive finite Borel measure on 
$\overline{\mathbb D}$, $\mathcal H=P^2(\nu)$, and $T=S_\nu$.
Then $\mathcal H_{0,T}=\{f\in P^2(\nu): \ f=0 \ \text {a.e. on}\ \ \mathbb T 
 \text { with respect to}\ \nu\}$, $\mathcal H_+^{(a)}=P^2(\nu|_{\mathbb T})$,
$$X_{T,+}\colon P^2(\nu)\to P^2(\nu|_{\mathbb T}), \ \ \ X_{T,+}f=f|_{\mathbb T}, \ \ f\in P^2(\nu),$$
 is the natural imbedding, and $T_+^{(a)}=S_{\nu|_{\mathbb T}}$.\end{lemma}

The proof of the following lemma is obvious and omitted.

\begin{lemma}\label{lem5.2} Suppose $\nu$ is a positive finite Borel measure on 
$\overline{\mathbb D}$, and $f\in P^2(\nu)$. Then there exists $\lambda\in \mathbb D$
such that $\|b_\lambda f\|=\|f\|$ if and only if $f(z)=0$ for a.e. $z\in\mathbb D$ 
with respect to $\nu$.\end{lemma}

\begin{corollary}\label{cor5.3} Suppose $\nu$ is a positive finite Borel measure on 
$\overline{\mathbb D}$, and $\lambda\in \mathbb D$. Then 
$$ P^2(\nu)\ominus\mathcal D_{b_\lambda(S_\nu)}\subset
\{f\in P^2(\nu):\ f(z)=0 \ \text{ for a.e. }\ z\in\mathbb D 
\ \text{ with respect to }  \nu\}$$
and
$$P^2(\nu)\ominus\mathcal D_{b_\lambda(S_\nu)^\ast}\subset
\{f\in P^2(\nu):\ f(z)=0 \ \text{ for a.e. }\ z\in\mathbb D 
\ \text{ with respect to }  \nu\}.$$\end{corollary}
\begin{proof} Let $P_+:L^2(\nu)\to P^2(\nu)$ be the orthogonal projection. 
It is easy to see that $(b_\lambda(S_\nu)f)(z)=b_\lambda(z)f(z)$ for a.e. $z\in \overline{\mathbb D}$ 
with respect to $\nu$, and $b_\lambda(S_\nu)^\ast f=P_+(\overline b_\lambda f)$, $f\in P^2(\nu)$.
If $f\in P^2(\nu)\ominus\mathcal D_{b_\lambda(S_\nu)}$, then, by \eqref{3.1}, $\|f\|=\| b_\lambda f\|$, 
and, by Lemma \ref{lem5.2}, $f(z)=0$ for a.e. $z\in\mathbb D$ with respect to $\nu$.
If $f\in P^2(\nu)\ominus\mathcal D_{b_\lambda(S_\nu)^\ast}$, then, by \eqref{3.1}, 
$\|f\|=\| P_+(\overline b_\lambda f)\|\leq\|\overline b_\lambda f\|=\|b_\lambda f\|$, 
and, by Lemma \ref{lem5.2}, $f(z)=0$ for a.e. $z\in\mathbb D$ with respect to $\nu$. \end{proof}

Denote by $m_2$ the normalized Lebesgue measure on the unit disk $\mathbb D$, for $-1<\alpha<\infty$
put $\text{\rm d}A_\alpha(z)=(\alpha+1)(1-|z|^2)^\alpha\text{\rm d}m_2(z)$. It is well known that 
the Bergman space $P^2(A_\alpha)$ has the following properties:
$f\in P^2(A_\alpha)$ if and only $f$ is an analytic function in $\mathbb D$ and $f\in L^2(A_\alpha)$, 
the functional $f\mapsto f(z)$ is bounded on $P^2(A_\alpha)$ for every $z\in\mathbb D$, 
and there exists a constant $C_\alpha>0$ (which depends on $\alpha$) such that
\begin{equation}\label{5.1} |f(z)|\leq C_\alpha\frac{\|f\|_{P^2(A_\alpha)}}{(1-|z|^2)^{1+\alpha/2}}, 
\ \ \ z\in\mathbb D, \end{equation}
(see, for example, {\cite[Sec. 1.1 and 1.2]{9}}). It is easy to see that 
$S_{A_\alpha}\in C_{00}$.


\begin{lemma}\label{lem5.4} {\cite[Lemma 4.2]{3}}  Let $-1<\alpha<\infty$. 
Then for every $f\in P^2(A_\alpha)$ and $\lambda\in\mathbb D$
$$\int_{\mathbb D}|b_\lambda f|^2\text{\rm d}A_\alpha\geq 
\frac{1}{\alpha+2}\int_{\mathbb D}|f|^2\text{\rm d}A_\alpha.$$\end{lemma}

\begin{corollary}\label{cor5.5} Let $-1<\alpha<\infty$, and let $\mu$ be a positive finite Borel measure on $\mathbb T$.
Then for every $f\in P^2(A_\alpha+\mu)$ and $\lambda\in\mathbb D$
$$\int_{\overline{\mathbb D}}|b_\lambda f|^2\text{\rm d}(A_\alpha+\mu)\geq 
\frac{1}{\alpha+2}\int_{\overline{\mathbb D}}|f|^2\text{\rm d}(A_\alpha+\mu).$$\end{corollary}
\begin{proof} Clearly, $P^2(A_\alpha+\mu)\subset P^2(A_\alpha)$. Let 
$f\in P^2(A_\alpha+\mu)$, and let $\lambda\in\mathbb D$. We have 
\begin{align*}\int_{\overline{\mathbb D}}  |b_\lambda &  f|^2\text{\rm d}(A_\alpha+\mu)  =
\int_{\mathbb D}|b_\lambda f|^2\text{\rm d}A_\alpha + 
\int_{\mathbb T}|b_\lambda f|^2\text{\rm d}\mu \\ &\geq 
\frac{1}{\alpha+2}\int_{\mathbb D}|f|^2\text{\rm d}A_\alpha + 
\int_{\mathbb T}|f|^2\text{\rm d}\mu 
\geq\frac{1}{\alpha+2}\int_{\overline{\mathbb D}}|f|^2\text{\rm d}(A_\alpha+\mu),\end{align*}
because of $|b_\lambda|=1$ on $\mathbb T$ and $1> 1/(\alpha+2)$ for  $-1<\alpha<\infty$. \end{proof}

Recall the following definition.

{\bf Definition.} Let $E$ be a closed subset of $\mathbb T$, and let $\{J_k\}_k$
be the collection of open arcs of $\mathbb T$ such that $J_k\cap J_\ell=\emptyset$
for $k\neq\ell$ and
$\mathbb T=E\cup\bigcup_kJ_k$. The set $E$ satisfies the {\it Carleson condition} 
if $\sum_km(J_k)\log m(J_k)>-\infty$.

\smallskip 

Let $w\in L^1(\mathbb T,m)$, $w\geq 0$ a.e. on $\mathbb T$. Then $P^2(wm)=L^2(wm)$ if and only if 
$\log w\not\in L^1(\mathbb T,m)$, and then $S_{wm}\cong U(\sigma)$, where $\sigma\subset\mathbb T$ is
a measurable set such that $wm$ and $m|_\sigma$ are mutually absolutely continuous.
If $ \log w\in L^1(\mathbb T,m)$, then there exists an outer function $\psi\in H^2$ such that 
$|\psi|^2=w$ a.e. on $\mathbb T$. Then 
\begin{equation}\label{5.2} \begin{aligned} P^2(wm) & =\frac{H^2}{\psi}=\Big\{\frac{h}{\psi}: \ h\in H^2\Big\}, 
\ \ \Big\|\frac{h}{\psi}\Big\|_{P^2(wm)}=\|h\|_{H^2},  \ h\in H^2,\\ 
\text{ and } \ \ S_{wm} & \cong S \end{aligned} \end{equation} 
(see, for example, {\cite[Ch. III.12]{6}} or {\cite[A.4.1.5]{14}}).

In Theorems \ref{th5.6} and \ref{th5.7}, we consider nontangential boundary values of functions from
$P^2(\mu)$ for some measures $\mu$. 
Nontangential boundary values of functions from
$P^t(\mu)$ with $1\leq t<\infty$ are considered in \cite{4} in relation to another questions, see also references therein, 
especially \cite{1},  \cite{12},  \cite{13},  \cite{18}, and  \cite{2}. In Theorems \ref{th5.6} and \ref{th5.7} we formulate particular cases of these results in the form convenient to our purpose.

Theorems \ref{th5.6} and  the main part of  Theorem \ref{th5.7}  were proved in {\cite[Sec. 2]{8}} for $\alpha=0$, 
but the proofs are the same in 
the case of $-1<\alpha\leq 0$ (because the estimate \eqref{5.1} involves the estimate 
$|f(z)|\leq C_\alpha\frac{\|f\|_{P^2(A_\alpha)}}{1-|z|^2}$ for
 $-1<\alpha\leq 0$, which is used in  {\cite[Sec. 2]{8}}), therefore, 
 the proofs of  Theorem \ref{th5.6}  and of the main part of Theorem \ref{th5.7}  are omitted.  
In addition,  to prove  Theorems \ref{th5.6} and \ref{th5.7}, one needs to apply the notion of isometric asymptote (see Sec. 2 of this paper and references therein). 

\begin{theorem}\label{th5.6} \cite{8}  Let $-1<\alpha\leq 0$, and let $E\subset\mathbb T$ be a closed set
such that  $0<m(E)<1$ and $E$ satisfies the Carleson condition. Then 
the functional $f\mapsto f(z)$ is bounded on $P^2(A_\alpha+m|_E)$ for every $z\in\mathbb D$.
Furthermore,    
 for   
 $f\in P^2(A_\alpha+m|_E)$ the restriction $f|_{\mathbb D}$ is  analytic on $\mathbb D$, 
$f|_{\mathbb D}$ has nontangential boundary values 
a.e. on $E$ with respect to $m$, which coincide with $f|_E$.
 Therefore, $S_{A_\alpha+m|_E}\in C_{10}$. 
Also, $I-S_{A_\alpha+m|_E}^\ast S_{A_\alpha+m|_E}$ is compact, and
$(S_{A_\alpha+m|_E})^{(a)}_+=U(E)$. Thus, $S_{A_\alpha+m|_E}$ 
is not similar to an isometry.\end{theorem}

\begin{theorem}\label{th5.7}  \cite{8} Let $-1<\alpha\leq 0$, and let 
$w\in L^1(\mathbb T,m)$.
Suppose that for every closed arc $J\subset\mathbb T\setminus\{1\}$ 
there exist two constants
 $0<c_J<C_J<\infty$ such that $c_J\leq w\leq C_J$ a.e. on $J$ 
 (with respect to $m$). Then 
the functional $f\mapsto f(z)$ is bounded on $P^2(A_\alpha+wm)$ 
for every $z\in\mathbb D$. Furthermore,  for 
 $f\in P^2(A_\alpha+wm)$ the restriction $f|_{\mathbb D}$ is  analytic on 
 $\mathbb D$, $f|_{\mathbb D}$ has nontangential boundary values 
a.e. on $\mathbb T$ with respect to $m$, which coincide with $f|_{\mathbb T}$.
 Therefore, $S_{A_\alpha+wm}\in C_{10}$. 
Also, $I-S_{A_\alpha+wm}^\ast S_{A_\alpha+wm}$ is compact, $(S_{A_\alpha+wm})^{(a)}_+=S_{wm}$,
and the canonical mapping which intertwines $S_{A_\alpha+wm}$ with $S_{wm}$ 
is the natural imbedding 
$$P^2(A_\alpha+wm)\to P^2(wm), \ \ f\mapsto f|_{\mathbb T}, \ f\in P^2(A_\alpha+wm).$$

Therefore, 

$\ \ \text{\rm (i)}$ $S_{A_\alpha+wm}\sim S$ if and only if $\log w\in L^1(\mathbb T,m)$;

$\ \text{\rm (ii)}$ $S_{A_\alpha+wm}$ is similar to an isometry if and only if $S_{A_\alpha+wm}\approx S$;

$\text{\rm (iii)}$ $S_{A_\alpha+wm}\approx S$ if and only if 
$\log w\in L^1(\mathbb T,m)$ and for every $h\in H^2$ there exists $f \in P^2(A_\alpha+wm)$ 
such that $f|_{\mathbb T}=h/\psi$ a.e.  on $\mathbb T$ (with respect to $m$), where $\psi\in H^2$  is
 an outer function such that 
$|\psi|^2=w$ a.e. on $\mathbb T$. \end{theorem}

{\bf Remark.} In the conditions of Theorem \ref{th5.7}, let $h,\psi\in H^2$, $\psi(z)\neq 0$ for every $z\in\mathbb D$,
$f\in P^2(A_\alpha+wm)$, and $f|_{\mathbb T} =h/\psi$ a.e. on $\mathbb T$ (with respect to $m$). Then 
$f(z)=h(z)/\psi(z)$ for every $z\in\mathbb D$. Indeed, set $g(z)=h(z)/\psi(z)$, $z\in\mathbb D$. Then 
$g$ is a function analytic on $\mathbb D$, and $g$ has nontangential boundary values $h(\zeta)/\psi(\zeta)$ 
for a.e. $\zeta\in\mathbb T$. Then $f-g$ is a function analytic on $\mathbb D$, and $f-g$ has zero nontangential 
boundary values  a.e. on $\mathbb T$. By Privalov's theorem (see, for example, {\cite[Theorem 8.1]{5}}), 
$f(z)=g(z)$ for every $z\in\mathbb D$. Therefore, if the conditions (iii) of Theorem \ref{th5.7} are fulfilled, 
then $ P^2(A_\alpha+wm)=H^2/\psi$ as the set, and the norms on these spaces are equivalent. 

\smallskip 

{\it Proof of Theorem \ref{th5.7}.} The main part of Theorem \ref{th5.7} is proved in {\cite[Sec. 2]{8}}.
Let $X$ be the imbedding,  
$$X\colon P^2(A_\alpha+wm)\to P^2(wm), \ \ Xf=f|_{\mathbb T}, \ \ f\in P^2(A_\alpha+wm).$$
Then $XS_{A_\alpha+wm}=S_{wm}X$. Since $S_{A_\alpha+wm}\in C_{1\cdot}$, $\ker X=\{0\}$, 
therefore, $X$ is a quasiaffinity which realizes the relation  $S_{A_\alpha+wm}\prec S_{wm}$.
If $\log w\in L^1(\mathbb T,m)$, then $S_{wm}\cong S$, therefore, $S_{A_\alpha+wm}\prec S$. 
Since $S_{A_\alpha+wm}$ is a cyclic contraction, $S_{A_\alpha+wm}\sim S$ by Theorem \ref{thb}(5). 
The ``if" part of (i) is proved.

The assumptions of the ``if" part of (iii) mean that $X P^2(A_\alpha+wm) = P^2(wm)$, see \eqref{5.2}.
Thus, $X$ realizes the relation $S_{A_\alpha+wm}\approx S_{wm}$, and, since $S_{wm}\cong S$,
the relation $S_{A_\alpha+wm}\approx S$ is proved.

Now suppose that $S_{A_\alpha+wm}\approx V$, where $V$ is an isometry. Then, by {\cite[Theorem 1]{11}}
(see Sec. 2 of this paper), $X P^2(A_\alpha+wm) = P^2(wm)$ and $V\cong S_{wm}$. Since 
$S_{A_\alpha+wm}\in C_{10}$, $S_{wm}\in C_{10}$. By {\cite[A.4.1.5]{14}} (see the description of $S_{wm}$ 
before \eqref{5.2} in this paper), $\log w\in L^1(\mathbb T,m)$ and $S_{wm}\cong S$. The parts (ii) and (iii)
are proved.

Now suppose that $S_{A_\alpha+wm}\sim S$. By  {\cite[Theorem 1]{11}}, see also {\cite[Ch. IX.1]{16}}, there exists 
an operator $Y\colon P^2(wm)\to H^2$ such that $YS_{wm}=SY$ and $\operatorname{clos}Y P^2(wm)= H^2$.
If $\log w\not\in L^1(\mathbb T,m)$, then $S_{wm}$ is unitary, and from the relations $Y^\ast S^\ast=S_{wm}^\ast Y^\ast$
and $\ker Y^\ast=\{0\}$ we conclude that $S^\ast\in C_{1\cdot}$, a contradiction. Therefore, 
$\log w\in L^1(\mathbb T,m)$. The ``only if" part of (i) is proved. \qed

\smallskip

The following lemma is a variant of  {\cite[Theorem 1.7]{9}}.

\begin{lemma}\label{lem5.8} Let $-1<\alpha<\infty$, and let $\beta\in\mathbb R$. 
Put $\varphi_\beta(z)=1/(1-z)^\beta$, $z\in\mathbb D$. Then 
$\varphi_\beta\in H^2$ if and only if $\beta<1/2$, and 
$\varphi_\beta\in P^2(A_\alpha)$ if and only if $\beta<1+\alpha/2$.\end{lemma}
\begin{proof} Put $v_n=2(\alpha+1)\int_0^1r^{2n+1}(1-r^2)^\alpha\text{\rm d}r$, $n\geq 0$.
Then $v_n=\frac{n!\Gamma(\alpha+2)}{\Gamma(\alpha+n+2)}$, where $\Gamma$ is the Gamma function, 
and
for every function $f$ analytic on $\mathbb D$
$$\int_{\mathbb D}|f|^2\text{\rm d} A_\alpha=\sum_{n=0}^\infty |\widehat f(n)|^2 v_n.$$
If $\beta\leq 0$, then $\varphi_\beta\in P^2(A_\alpha)$ for every $\alpha$, $-1<\alpha<\infty$. Suppose $\beta>0$. 
Since $\widehat \varphi_{\beta}(n)=\frac{\Gamma(\beta+n)}{n!\Gamma(\beta)}$, $n\geq 0$, we have that 
$\varphi_\beta\in P^2(A_\alpha)$ if and only if the series 
$\sum_{n=0}^\infty
\frac{\Gamma(\beta+n)^2}{n!\Gamma(\alpha+n+2)}$ converges. By Stirling's formula,
$$\frac{\Gamma(\beta+n)^2}{n!\Gamma(\alpha+n+2)}\sim (n+1)^{2\beta-\alpha-3} \ \ \text { as } \ n\to\infty.$$
Therefore,  the series 
$\sum_{n=0}^\infty
\frac{\Gamma(\beta+n)^2}{n!\Gamma(\alpha+n+2)}$ converges if and only if $\beta<1+\alpha/2$.

The first statement of the lemma can be proved similarly.  \end{proof}

\begin{lemma}\label{lem5.9} Let $-1<\alpha\leq 0$, and let $\beta<-1-\alpha$. Then $S_{A_\alpha+|\varphi_\beta| m}\sim S$,
but $S_{A_\alpha+|\varphi_\beta| m}$ is not similar to an isometry.\end{lemma}
\begin{proof} 
By Theorem \ref{th5.7}(i), 
$S_{A_\alpha+|\varphi_\beta|m}\sim S$. Put $\psi=\varphi_{\beta/2}$. Clearly, $|\psi|^2=|\varphi_\beta|$.
By \eqref{5.2}, $P^2(|\varphi_\beta|m)=H^2/\psi$. By Theorem \ref{th5.7}(iii), if $S_{A_\alpha+|\varphi_\beta|m}$ 
is similar to an isometry, then $H^2/\psi = P^2(A_\alpha+|\varphi_\beta|m)\subset P^2(A_\alpha)$.
Take $\gamma$,  $1+\alpha/2+\beta/2\leq\gamma<1/2$. Put $h=\varphi_\gamma$. Then $h\in H^2$, 
and $h/\psi=\varphi_{\gamma-\beta/2}\not\in P^2(A_\alpha)$
by Lemma \ref{lem5.8}. Therefore, $S_{A_\alpha+|\varphi_\beta|m}$ is not similar to an isomertry.  \end{proof}

\begin{corollary}\label{cor5.10} Let $0<\delta<1$, and let $E\subset\mathbb T$ be a closed set 
satisfying the Carleson condition and such that $0<m(E)<1$. Then there exist operator-valued 
inner functions $F_k$, 
such that $F_k$ satisfy $\eqref{1.2}$, and $\eqref{1.3}$ with $\delta$,  and $F_k$ are not 
left invertible,   
$k=1,2,3$. Also, $F_1$, $F_2$, $F_3$ satisfy $\eqref{1.4}$, $\eqref{1.6}$, $\eqref{1.5}$, respectively, 
$F_3$ has an outer left scalar multiple, and  $I-F_3(z)^\ast F_3(z)\in\frak S_1$ for every $z\in\mathbb D$.
\end{corollary}

{\bf Remark.} Since $F_3$ is not left invertible, $F_3$ does not satisfy \eqref{1.7}, see \cite{22}.

\smallskip 

{\it Proof of Corollary \ref{cor5.10}.} Put $\alpha=1/\max(\delta^2, 1/2)-2$, then  $-1<\alpha\leq 0$.  
Take $\beta<-1-\alpha$.  Put 
$$\mathcal H_1=P^2(A_\alpha), \ \ \mathcal H_2=P^2(A_\alpha+m|_E), \ \ \mathcal H_3=P^2(A_\alpha+|\varphi_\beta|m), $$
$$ T_1=S_{A_\alpha}, \ \ T_2=S_{A_\alpha+m|_E}, \ \ T_3=S_{A_\alpha+|\varphi_\beta|m}.$$
Clearly, $T_k$ are cyclic contractions, $k=1,2,3$, $T_1\in C_{00}$, and $T_k\in C_{10}$ by 
Theorems \ref{th5.6} and \ref{th5.7},
$k=2,3$. By Corollary \ref{cor5.5}, \begin{equation}\label{5.3} \|b_\lambda(T_k)f\|^2\geq \delta^2\|f\|^2 \ \text{ for every } \ \lambda\in\mathbb D, \ 
\ f\in\mathcal H_k, \ \ k=1,2,3. \end{equation}

By Theorems  \ref{th5.6} and \ref{th5.7}, the functionals 
$$ f\mapsto f(z), \ \ \mathcal H_k\to\mathbb C,$$
are bounded for every $z\in\mathbb D$, $k=1,2,3$.
Therefore, \begin{equation}\label{5.4} \dim \mathcal H_k\ominus b_\lambda(T_k)\mathcal H_k =1 \ \text{ for every } \ \lambda\in\mathbb D, \ 
\ \ k=1,2,3. \end{equation}
Indeed, let $k$ be fixed, and let $\lambda\in\mathbb D$. Then there exists $g_\lambda\in \mathcal H_k$ such that
$f(\lambda)=(f,g_\lambda)$ for every $f\in \mathcal H_k$. Since $(b_\lambda(T_k)f)(z)=b_\lambda(z)f(z)$ for every
$z\in\mathbb D$, $f\in \mathcal H_k$, it is clear that $g_\lambda\in\mathcal H_k\ominus b_\lambda(T_k)\mathcal H_k$. 
Since $T_k$ is cyclic, $T_k-\lambda I$ is cyclic, too, therefore, 
 $\dim \mathcal H_k\ominus b_\lambda(T_k)\mathcal H_k = \dim \mathcal H_k\ominus (T_k-\lambda I)\mathcal H_k\leq 1$. 
The equality \eqref{5.4} is proved.

Now find the unitary asymptotes of $T_k$, $k=1,2,3$, see Sec. 2 of this paper and references therein.
 Since $T_1\in C_{00}$, $T_1^{(a)}=\mathbb O$. By Theorem \ref{th5.6}, $T_2^{(a)}=U(E)$. By Lemma \ref{lem5.9},
$T_3\sim S$, therefore, $T_3^{(a)}=U(\mathbb T)$, the bilateral shift of multiplicity 1.
Since $T^{(a)}\cong U(\omega_T)$ for every cyclic completely nonunitary contraction $T$, where 
$\omega_T$ is defined in \eqref{3.5}, we conclude that 
\begin{equation}\label{5.5} \omega_{T_1}=\emptyset, \ \ \  \omega_{T_2}=E, \ \ \ \omega_{T_3}=\mathbb T.  \end{equation} 
Also, $T_1$ is not similar to an isometry, because $T_1\in C_{00}$, and $T_2$ and $T_3$ are 
not similar to an isometry by Theorem \ref{th5.6} and Lemma \ref{lem5.9}, respectively.

Now put $F_k=\Theta_{T_k}$, that is, $F_k$ is the characteristic function of the contraction $T_k$, 
 $k=1,2,3$, see Sec. 3 of this paper and references therein. Since $T_k\in C_{\cdot 0}$, $F_k$ are inner.
By \eqref{5.3} and Lemma \ref{lem3.2}(i), $F_k$ satisfy \eqref{1.3} with $\delta$. By \eqref{5.4} and 
Lemma \ref{lem3.2}(ii), $F_k$ satisfy \eqref{1.2}. $F_k$ are not left invertible,  because of $T_k$ are not 
similar to an isometry, see Theorem \ref{tha}. $F_1$, $F_2$, $F_3$ satisfy \eqref{1.4}, \eqref{1.6}, \eqref{1.5}, respectively, 
because of \eqref{3.7} and \eqref{5.5}. Since $T_3\sim S$ (by Lemma \ref{lem5.9}), $F_3$  has an outer 
left scalar multiple
by Theorem \ref{thb}(6), and  $I-F_3(z)^\ast F_3(z)\in\frak S_1$ for every $z\in\mathbb D$ 
by Lemma \ref{lem3.2}(iii) and Theorem \ref{thb}(2).
  \qed


\begin{thebibliography}{HD}


\normalsize
\baselineskip=11pt

\bibitem[1]{1} J. R. Akeroyd, {\it Another look at some index theorems for the shift}, Indiana Univ. Mat. J. {\bf 50} (2001), 
705--718.  

\ 

\bibitem[2]{2} J. R. Akeroyd, {\it A note on harmonic measure}, Comput. Methods Funct. Theory {\bf 7} (2007), no. 1, 91--104. 

\
\bibitem[3]{3} A. Aleman, S. Richter, and C. Sundberg, {\it
Invariant subspaces for the backward shift on Hilbert spaces of analytic functions with regular norm},  in:  Bergman spaces and related topics in complex analysis. Contemp. Math. {\bf 404} (2006) 
AMS, Providence, RI, 1--25.

 \

\bibitem[4]{4} A. Aleman, S. Richter, and  C. Sundberg, {\it Nontangential limits in $\mathcal P^t(\mu)$-spaces
and the index of invariant subspaces},  Ann. of Math. (2)  {169} (2009), 449--490.
 
\

\bibitem[5]{5} E. F. Collingwood and A. J. Lohwater, {\it The theory of cluster sets},
Cambridge Univ. Press, 1966. 

 \ 

\bibitem[6]{6} J. B. Conway, {\it The theory of subnormal operators}, Math. Surveys 
and Monographs, Amer. Math. Soc., v. 36, 1991.

\

\bibitem[7]{7} M. F. Gamal', {\it On contractions that are quasiaffine transforms 
of unilateral shifts},   Acta Sci. Math. (Szeged) 
{\bf 74} (2008), 755--765.

\

\bibitem[8]{8} M. F. Gamal', {\it On contractions with compact defects}, 
 Zap. Nauchn. Sem. POMI  
{\bf 366} (2009), 13--41 (Russian); 
translation in J. Math. Sci. (N. Y.) {\bf 165} (2010), no. 4, 435--448.

\

\bibitem[9]{9} H. Hedenmalm, B. Korenblum, and  K. Zhu, {\it Theory of Bergman spaces.} 
Graduate Texts in Mathematics, {\bf 199}, New York, 2000.

\

\bibitem[10]{10}  V. V. Kapustin and A. V. Lipin, {\it  Operator algebras and lattices of
     invariant subspaces, I, II},  Zap. Nauchn. Semin. LOMI, {\bf 178} (1989), 23--56; {\bf 190} (1991),
        110--147 (Russian);  English translation in: J. Soviet Math.,
	 {\bf 61} (1992), 1963--1981,  J. Math. Sci.,
	  {\bf 71} (1994), 2240--2262.

\

\bibitem[11]{11} L. K\'erchy, {\it Isometric asymptotes of power bounded operators.} -- 
 Indiana Univ. Math. J. 
{\bf 38} (1989), 173--188. 

\

\bibitem[12]{12} T. L. Miller and R. C. Smith, {\it 
Nontangential limits of functions in some $P^2(\mu)$ spaces}, Indiana Univ. Math. J. {\bf 39} (1990),  no.1, 19--26.

\

\bibitem[13]{13}T. L. Miller, W.  Smith, and L. Yang, {\it
Bounded point evaluations for certain $P^t(\mu)$ spaces}, Illinois J. Math. {\bf 43} (1999), no.1, 131--150.

\

\bibitem[14]{14}
 N. K. Nikolski, {\it  Operators, functions, and systems: an easy reading.
 Volume I:
Hardy, Hankel, and Toeplitz, Volume II: Model operators and systems}, 
Math. Surveys and Monographs {\bf 92}, AMS, 2002.  

\

\bibitem[15]{15} B. Sz.-Nagy and  C. Foias, {\it On the structure of intertwining operators},  
Acta Sci. Math. (Szeged) {\bf 35} (1973), 225--254.

\

\bibitem[16]{16} B. Sz.-Nagy,  C. Foias, H. Bercovici, and L. K\'erchy, 
{\it Harmonic analysis of operators on Hilbert space}, 
 Springer, New York, 2010. 

\

\bibitem[17]{17}   K. Takahashi, {\it On quasiaffine transforms of unilateral shifts,}  
   Proc. Amer. Math. Soc. {\bf 100} (1987), 683--687.

\

\bibitem[18]{18} J. E. Thomson and L.  Yang, {\it 
Invariant subspaces with the codimension one property in $L^t(\mu)$}, 
Indiana Univ. Math. J. {\bf 44} (1995), no.4, 1163--1173.

\


\bibitem[19]{19} V. A. Tolokonnikov, {\it Estimates in the Carleson corona theorem, 
ideals of the algebra $H^\infty$, 
a problem of Sz.-Nagy},   Zap. Nauchn. Semin. LOMI {\bf 113} (1981), 178--198 (Russian); English translation in:  J. Soviet. Math. {\bf 22} (1983), 1814--1828. 

\

\bibitem[20]{20} S. R.  Treil', {\it Angles between co-invariant subspaces, and the operator corona problem. 
The Sz\"okefalvi-Nagy problem},   Dokl. Akad. Nauk SSSR, {\bf 302} (1988), no. 5, 1063--1068 (Russian); English translation in 
 Soviet Math. Dokl., {\bf  38} (1989), no. 2, 394--399.

\

\bibitem[21]{21} S. Treil, {\it Geometric methods in spectral theory of vector-valued 
functions: Some recent results}, in: 
 Toeplitz operators and spectral function theory, 
Oper. Theory, Adv. Appl. {\bf 42} (1989), 209-280. 

\

\bibitem[22]{22} S. Treil, {\it An operator corona theorem},  Indiana Univ. Math. J.
 {\bf 53} (2004), 1763--1781.

\

\bibitem[23]{23} S. Treil, {\it Lower bounds in the matrix corona theorem and 
the codimension one conjecture},  
 Geom. Funct. Anal. {\bf 14} (2004), 1118--1133.

\

\bibitem[24]{24}   M. Uchiyama, {\it Contractions and unilateral shifts},  
     Acta Sci. Math. (Szeged) {\bf 46} (1983), 345--356.

\

\bibitem[25]{25} M. Uchiyama, {\it Curvatures and similarity of operators with 
holomorphic  eigenvectors},  Trans. Amer. Math. Soc., {\bf 319} (1990), 405--415.

\end{thebibliography}
\end{document}